\documentclass[11pt]{amsart}
\usepackage{amssymb}
\usepackage{amsmath}
\usepackage[active]{srcltx}
\usepackage{t1enc}
\usepackage[latin2]{inputenc}
\usepackage{verbatim}
\usepackage{amsmath,amsfonts,amssymb,amsthm}
\usepackage[mathcal]{eucal}
\usepackage{enumerate}
\usepackage[centertags]{amsmath}
\usepackage{graphics}

\newtheorem{theorem}{Theorem}

\newtheorem{lemma}{Lemma}

\begin{document}
\author{George Tephnadze}
\title[Fejér means]{On the convergence of Fejér means of Walsh-Fourier
series in the space $H_{p}$}
\address{G. Tephnadze, Department of Mathematics, Faculty of Exact and
Natural Sciences, Tbilisi State University, Chavchavadze str. 1, Tbilisi
0128, Georgia and Department of Engineering Sciences and Mathematics, Lule%
\aa {} University of Technology, SE-971 87, Lule\aa {}, Sweden.}
\thanks{The research was supported by Shota Rustaveli National Science
Foundation grant no.52/54 (Bounded operators on the martingale Hardy spaces).%
}
\email{giorgitephnadze@gmail.com}
\date{}
\maketitle

\begin{abstract}
The main aim of this paper is to find the necessary and sufficient
conditions for a modulus of continuity of a martingale $F\in H_{p},$ for
which Fejér means convergence in $H_{p}$-norm, when $0<p\leq 1/2.$
\end{abstract}

\date{}



\section{INTRODUCTION}

Weisz \cite{We2} considered the norm convergence of Fejér means of
Walsh-Fourier series and proved that
\begin{equation}
\left\Vert \sigma _{k}F\right\Vert _{H_{p}}\leq c_{p}\left\Vert F\right\Vert
_{H_{p}},\left( F\in H_{p},\text{ \ }p>1/2\right) .  \label{markW1}
\end{equation}

Goginava in \cite{Goginava} \ (see also \cite{tep1}) proved that there
exists a martingale $F\in H_{p}$ $\left( 0<p\leq 1/2\right) ,$ such that
\begin{equation*}
\sup_{n}\left\Vert \sigma _{n}F\right\Vert _{p}=+\infty .
\end{equation*}

Weisz \cite{We3} also considered Fejér means on the subsequence $\left\{
2^{n}:n\geq 0\right\} $ and proved that
\begin{equation}
\left\Vert \sigma _{2^{n}}F\right\Vert _{H_{p}}\leq c_{p}\left\Vert
F\right\Vert _{H_{p}},\text{\ }\left( F\in H_{p},\text{ \ }p>0\right) .
\label{77}
\end{equation}

The author \cite{tep6} proved that if $F\in H_{p}$ $\left( 0<p<1/2\right) $
and
\begin{equation*}
\omega _{H_{p}}\left( 1/2^{n},F\right) =o\left( 1/2^{n\left( 1/p-2\right)
}\right) ,\text{ as \ }n\rightarrow \infty ,
\end{equation*}%
then
\begin{equation}
\left\Vert \sigma _{n}F-F\right\Vert _{p}\rightarrow 0,\text{ when }%
n\rightarrow \infty .  \label{fe1}
\end{equation}

Furthermore, it was shown that there exists a martingale $F\in H_{p}$ $%
\left( 0<p<1/2\right) ,$\ for which
\begin{equation*}
\omega _{H_{p}}\left( 1/2^{n},F\right) =O\left( 1/2^{n\left( 1/p-2\right)
}\right) ,\text{ as \ }n\rightarrow \infty
\end{equation*}%
and (\ref{fe1}) does not holds, as\thinspace \thinspace \thinspace $%
n\rightarrow \infty .$

The author \cite{tep6} also consider endpoint case $p=1/2$ and proved that
if $F\in H_{1/2}$ and
\begin{equation}
\omega _{H_{1/2}}\left( 1/2^{n},F\right) =o\left( 1/n^{2}\right) ,\text{ as
\ }n\rightarrow \infty ,  \label{fe0}
\end{equation}%
then
\begin{equation*}
\left\Vert \sigma _{n}F-F\right\Vert _{1/2}\rightarrow 0,\text{ when }%
n\rightarrow \infty .
\end{equation*}

Moreover, condition (\ref{fe0}) is sharp in the same sense.

The results for summability of Fejér means of Walsh-Fourier series can be
found in \cite{tb}, \cite{gog8}, \cite{Mor,na, PS}, \cite{We2}.

The main aim of this paper is to generalize inequality (\ref{markW1}) and
find necessary and sufficient conditions on sequences $\left\{ n_{k}:k\geq
0\right\} ,$ for which%
\begin{equation}
\left\Vert \sigma _{n_{k}}F-F\right\Vert _{p}\rightarrow 0,\text{ \ \ when \
}k\rightarrow \infty ,\text{ \ }\left( F\in H_{p},\text{ \ }0<p\leq
1/2\right) ,  \label{fe2}
\end{equation}

Moreover, we prove some estimates for the Fejér means and show the sharpness
of them, when $0<p\leq 1/2$ (see Theorems \ref{theorem1} and \ref{theorem2}%
). By applying this estimates we find the necessary and sufficient
conditions for a modulus of continuity of a martingale $F\in H_{p},$ for
which (\ref{fe2}) holds, when $0<p\leq 1/2$ (see Theorems \ref{theorem3} and %
\ref{theorem4}).

\section{Definitions and Notations}

Let $\mathbb{N}_{+}$ denote the set of the positive integers, $\mathbb{N}:=%
\mathbb{N}_{+}\cup \{0\}.$ Denote by $Z_{2}$ the discrete cyclic group of
order 2, that is $Z_{2}:=\{0,1\},$ where the group operation is the modulo 2
addition and every subset is open. The Haar measure on $Z_{2}$ is given so
that the measure of a singleton is 1/2.

Define the group $G$ as the complete direct product of the group $Z_{2},$
with the product of the discrete topologies of $Z_{2}$`s. The elements of $G$
are represented by sequences $x:=(x_{0},x_{1},...,x_{j},...),$ where $%
x_{k}=0\vee 1.$

It is easy to give a base for the neighborhood of $x\in G$
\begin{equation*}
I_{0}\left( x\right) :=G,\text{ \ }I_{n}(x):=\{y\in
G:y_{0}=x_{0},...,y_{n-1}=x_{n-1}\}\text{ }(n\in \mathbb{N}).
\end{equation*}

Denote $I_{n}:=I_{n}\left( 0\right) ,$ $\overline{I_{n}}:=G$ $\backslash $ $%
I_{n}$ and $e_{n}:=\left( 0,...,0,x_{n}=1,0,...\right) \in G,$ for $n\in
\mathbb{N}$. Then it is easy to show that
\begin{equation}
\overline{I_{M}}=\underset{i=0}{\bigcup\limits^{M-1}}I_{i}\backslash
I_{i+1}=\left( \overset{M-2}{\underset{k=0}{\bigcup }}\overset{M-1}{\underset%
{l=k+1}{\bigcup }}I_{l+1}\left( e_{k}+e_{l}\right) \right) \bigcup \left(
\underset{k=0}{\bigcup\limits^{M-1}}I_{M}\left( e_{k}\right) \right) .
\label{1}
\end{equation}

If $n\in \mathbb{N},$ then every $n$ can be uniquely expressed as $%
n=\sum_{k=0}^{\infty }n_{j}2^{j},$ where $n_{j}\in Z_{2}$ $~(j\in \mathbb{N}%
) $ and only a finite numbers of $n_{j}$ differ from zero.

Let
\begin{equation*}
\left[ n\right] :=\min \{j\in \mathbb{N},n_{j}\neq 0\}\text{ \ \ and \ \ \ }%
\left\vert n\right\vert :=\max \{j\in \mathbb{N},n_{j}\neq 0\},
\end{equation*}
that is $2^{\left\vert n\right\vert }\leq n\leq 2^{\left\vert n\right\vert
+1}.$ Set
\begin{equation*}
d\left( n\right) =\left\vert n\right\vert -\left[ n\right] ,\text{ \ for \
all \ \ }n\in \mathbb{N}.
\end{equation*}

Define the variation of $n\in \mathbb{N},$ with binary coefficients $\left(
n_{k},k\in \mathbb{N}\right) ,$ by
\begin{equation}
V\left( n\right) =n_{0}+\sum_{k=1}^{\infty }\left\vert
n_{k}-n_{k-1}\right\vert .  \label{var}
\end{equation}

Every $n\in \mathbb{N}$ can be also represented as $n=%
\sum_{i=1}^{r}2^{n_{i}},n_{1}>n_{2}>...n_{r}\geq 0.$ For such representation
of $n\in \mathbb{N},$ let denote numbers
\begin{equation*}
n^{\left( i\right) }=2^{n_{i+1}}+...+2^{n_{r}},i=1,...,r.
\end{equation*}

The norms (or quasi-norm) of the spaces $L_{p}(G)$ and $L_{p,\infty }\left(
G\right) ,$ $\left( 0<p<\infty \right) $ are respectively defined by
\begin{equation*}
\left\Vert f\right\Vert _{p}^{p}:=\int_{G}\left\vert f\right\vert ^{p}d\mu
,\ \ \ \ \left\Vert f\right\Vert _{L_{p,\infty }(G)}^{p}:=\sup_{\lambda
>0}\lambda ^{p}\mu \left( f>\lambda \right) <+\infty ,
\end{equation*}

The $k$-th Rademacher function is defined by
\begin{equation*}
r_{k}\left( x\right) :=\left( -1\right) ^{x_{k}}\text{\qquad }\left( \text{ }%
x\in G,\text{ }k\in \mathbb{N}\right) .
\end{equation*}

Now, define the Walsh system $w:=(w_{n}:n\in \mathbb{N})$ on $G$ as:
\begin{equation*}
w_{n}(x):=\overset{\infty }{\underset{k=0}{\Pi }}r_{k}^{n_{k}}\left(
x\right) =r_{\left\vert n\right\vert }\left( x\right) \left( -1\right) ^{%
\underset{k=0}{\overset{\left\vert n\right\vert -1}{\sum }}n_{k}x_{k}}\text{%
\qquad }\left( n\in \mathbb{N}\right) .
\end{equation*}

The Walsh system is orthonormal and complete in $L_{2}\left( G\right) $ (see
\cite{sws})$.$

If $f\in L_{1}\left( G\right) ,$ we can establish Fourier coefficients,
partial sums of Fourier series, Fejér means, Dirichlet and Fejér kernels in
the usual manner:
\begin{equation*}
\widehat{f}\left( n\right) :=\int_{G}fw_{n}d\mu ,\,\,\left( n\in \mathbb{N}%
\right) ,\text{ \ }S_{n}f:=\sum_{k=0}^{n-1}\widehat{f}\left( k\right) w_{k},%
\text{\ }\left( n\in \mathbb{N}_{+},S_{0}f:=0\right) ,
\end{equation*}%
\begin{equation*}
\sigma _{n}f:=\frac{1}{n}\sum_{k=1}^{n}S_{k}f,\text{ \ \ }%
D_{n}:=\sum_{k=0}^{n-1}w_{k\text{ }},\text{ \ }K_{n}:=\frac{1}{n}\overset{n}{%
\underset{k=1}{\sum }}D_{k}\text{ },\text{ \ }\left( \text{ }n\in \mathbb{N}%
_{+}\text{ }\right) .
\end{equation*}

Recall that (see (\cite{sws}))%
\begin{equation}
D_{2^{n}}\left( x\right) =\left\{
\begin{array}{ll}
2^{n} & \,\text{if\thinspace \thinspace \thinspace }x\in I_{n} \\
0\, & \ \,\text{if}\,\,x\notin I_{n}.%
\end{array}%
\right.  \label{1dn}
\end{equation}

Let $n=\sum_{i=1}^{r}2^{n_{i}},$ $n_{1}>n_{2}>...>n_{r}\geq 0.$ Then (see
\cite{G-E-S} and \cite{sws})
\begin{equation}
nK_{n}=\sum_{A=1}^{r}\left( \underset{j=1}{\overset{A-1}{\prod }}%
w_{2^{n_{j}}}\right) \left( 2^{n_{A}}K_{2^{n_{A}}}-n^{\left( A\right)
}D_{2^{n_{A}}}\right) .  \label{9a}
\end{equation}

The $\sigma $-algebra, generated by the intervals $\left\{ I_{n}\left(
x\right) :x\in G\right\} $ will be denoted by $\zeta _{n}$ $\left( n\in
\mathbb{N}\right) .$ Denote by $F=\left( F_{n},n\in \mathbb{N}\right) $ a
martingale with respect to $\zeta _{n}$ $\left( n\in \mathbb{N}\right) .$
(for details see e.g. \cite{We1}). The maximal function of a martingale $F$
is defined by
\begin{equation*}
F^{\ast }:=\sup_{n\in \mathbb{N}}\left\vert F_{n}\right\vert .
\end{equation*}

In case $f\in L_{1}\left( G\right) ,$ the maximal functions are also be
given by
\begin{equation*}
f^{\ast }\left( x\right) :=\sup\limits_{n\in \mathbb{N}}\left( \frac{1}{\mu
\left( I_{n}\left( x\right) \right) }\left\vert \int_{I_{n}\left( x\right)
}f\left( u\right) d\mu \left( u\right) \right\vert \right) .
\end{equation*}

For $0<p<\infty ,$ Hardy martingale spaces $H_{p}$ $\left( G\right) $
consist all martingales, for which
\begin{equation*}
\left\Vert F\right\Vert _{H_{p}}:=\left\Vert F^{\ast }\right\Vert
_{p}<\infty .
\end{equation*}

The best approximation of $f\in L_{p}(G)$ $(1\leq p\in \infty )$ is defined
as%
\begin{equation*}
E_{n}\left( f,L_{p}\right) :=\inf_{\psi \in \emph{p}_{n}}\left\Vert f-\psi
\right\Vert _{p},
\end{equation*}%
where $\emph{p}_{n}$ is set of all Walsh polynomials of order less than $%
n\in \mathbb{N}$.

The modulus of continuity of the function $\ f\in L_{p}\left( G\right) ,$ is
defined by

\begin{equation*}
\omega _{p}\left( 1/2^{n},f\right) :=\sup\limits_{h\in I_{n}}\left\Vert
f\left( \cdot +h\right) -f\left( \cdot \right) \right\Vert _{p}.
\end{equation*}

The concept of modulus of continuity in $H_{p}(G)$ $\left( p>0\right) $ is
defined in the following way
\begin{equation*}
\omega _{H_{p}}\left( 1/2^{n},f\right) :=\left\Vert f-S_{2^{n}}f\right\Vert
_{H_{p}}.
\end{equation*}

Since $\left\Vert f\right\Vert _{H_{p}}\sim \left\Vert f\right\Vert _{p}$,
for $p>1$, we obtain that $\omega _{H_{p}}\left( 1/2^{n},f\right) \sim
\left\Vert f-S_{2^{n}}f\right\Vert _{p},$ \ \ $p>1.$ On the other hand,
there are strong connection among this definitions:%
\begin{equation*}
\omega _{p}\left( 1/2^{n},f\right) /2\leq \left\Vert f-S_{2^{n}}f\right\Vert
_{p}\leq \omega _{p}\left( 1/2^{n},f\right) ,
\end{equation*}%
and%
\begin{equation*}
\left\Vert f-S_{2^{n}}f\right\Vert _{p}/2\leq E_{2^{n}}\left( f,L_{p}\right)
\leq \left\Vert f-S_{2^{n}}f\right\Vert _{p}.
\end{equation*}

A bounded measurable function $a$ is p-atom, if there exists an interval $I$%
, such that%
\begin{equation*}
\int_{I}ad\mu =0,\text{ \ }\left\Vert a\right\Vert _{\infty }\leq \mu \left(
I\right) ^{-1/p},\text{ \ \ supp}\left( a\right) \subset I.
\end{equation*}%
\qquad

It is easy to check that for every martingale $F=\left( F_{n},n\in \mathbb{N}%
\right) $ and every $k\in \mathbb{N}$ the limit

\begin{equation*}
\widehat{F}\left( k\right) :=\lim_{n\rightarrow \infty }\int_{G}F_{n}\left(
x\right) w_{k}\left( x\right) d\mu \left( x\right)
\end{equation*}%
exists and it is called the $k$-th Walsh-Fourier coefficients of $F.$

The Walsh-Fourier coefficients of $f\in L_{1}\left( G\right) $ are the same
as those of the martingale $\left( S_{M_{n}}\left( f\right) :n\in \mathbb{N}%
\right) $ obtained from $f$.

For the martingale $F=\sum_{n=0}^{\infty }\left( F_{n}-F_{n-1}\right) $ $\ $%
the conjugate transforms are defined as $\widetilde{F^{\left( t\right) }}=%
\overset{\infty }{\underset{n=0}{\sum }}r_{n}\left( t\right) \left(
F_{n}-F_{n-1}\right) ,$ where $t\in G$ is fixed. Note that $\widetilde{%
F^{\left( 0\right) }}=F.$ As is well known (see \cite{We1})
\begin{equation}
\left\Vert \widetilde{F^{\left( t\right) }}\right\Vert _{H_{p}}=\left\Vert
F\right\Vert _{H_{p}},\text{ \ \ }\left\Vert F\right\Vert _{H_{p}}^{p}\sim
\int_{G}\left\Vert \widetilde{F^{\left( t\right) }}\right\Vert _{p}^{p}dt,%
\text{ \ \ \ }\widetilde{\left( \sigma _{n}F\right) ^{\left( t\right) }}%
=\sigma _{n}\widetilde{F^{\left( t\right) }}.  \label{5.1}
\end{equation}

\section{Formulation of Main Results}

\begin{theorem}
\label{theorem1}a) Let $F\in H_{1/2}.$ Then there exists an absolute
constant $c,$ such that%
\begin{equation*}
\left\Vert \sigma _{n}F\right\Vert _{H_{1/2}}\leq cV^{2}\left( n\right)
\left\Vert F\right\Vert _{H_{1/2}}.
\end{equation*}

\textit{b) Let} $\left\{ n_{k}:k\geq 0\right\} $ \textit{be subsequence of
positive integers} $\mathbb{N}_{+},$ \textit{such that} $\sup_{k}V\left(
n_{k}\right) =\infty $ \textit{and} $\Phi :\mathbb{N}_{+}\rightarrow \lbrack
1,\infty )$ \textit{be any nondecreasing, nonnegative function, satisfying
conditions} $\Phi \left( n\right) \uparrow \infty $ \textit{and}
\begin{equation}
\overline{\underset{k\rightarrow \infty }{\lim }}\frac{V^{2}\left(
n_{k}\right) }{\Phi \left( n_{k}\right) }=\infty .  \label{30}
\end{equation}%
\textit{Then there exists a martingale} $F\in H_{1/2},$ \textit{such that}
\begin{equation*}
\underset{k\in \mathbb{N}}{\sup }\left\Vert \frac{\sigma _{n_{k}}F}{\Phi
\left( n_{k}\right) }\right\Vert _{1/2}=\infty .
\end{equation*}
\end{theorem}

\begin{theorem}
\label{theorem2}a) Let $0<p<1/2,$ $F\in H_{p}.$ Then there exists an
absolute constant $c_{p}$, defending only on $p$, such that
\begin{equation*}
\text{ }\left\Vert \sigma _{n}F\right\Vert _{H_{p}}\leq c_{p}2^{d\left(
n\right) \left( 1/p-2\right) }\left\Vert F\right\Vert _{H_{p}}.
\end{equation*}

\textit{b) Let }$0<p<1/2$ \textit{and} $\Phi \left( n\right) $ \textit{be
any nondecreasing function,\ such that }
\begin{equation}
\sup_{k}d\left( n_{k}\right) =\infty ,\text{ \ \ }\overline{\underset{%
k\rightarrow \infty }{\lim }}\frac{2^{d\left( n_{k}\right) \left(
1/p-2\right) }}{\Phi \left( n_{k}\right) }=\infty .  \label{31aaa}
\end{equation}%
\textit{Then there exists a martingale }$F\in H_{p},$\textit{\ such that}
\begin{equation*}
\underset{k}{\sup }\left\Vert \frac{\sigma _{n_{k}}F}{\Phi \left(
n_{k}\right) }\right\Vert _{L_{p,\infty }}=\infty .
\end{equation*}
\end{theorem}

\begin{theorem}
\label{theorem3}a) Let $F\in H_{1/2}$, $\sup_{k}V\left( n_{k}\right) =\infty
$ and
\begin{equation}
\omega _{H_{p}}\left( 1/2^{\left\vert n_{k}\right\vert },F\right) =o\left(
1/V^{2}\left( n_{k}\right) \right) ,\text{ as \ }k\rightarrow \infty .
\label{cond}
\end{equation}%
Then (\ref{fe2}) holds, for $p=1/2.$

b) Let $\sup_{k}V\left( n_{k}\right) =\infty .$ Then there exists a
martingale $f\in H_{1/2}(G),$\ \ for which
\begin{equation}
\omega _{H_{1/2}}\left( 1/2^{\left\vert n_{k}\right\vert },F\right) =O\left(
1/V^{2}\left( n_{k}\right) \right) ,\text{ \ as \ }k\rightarrow \infty
\label{cond2}
\end{equation}%
and
\begin{equation}
\left\Vert \sigma _{n_{k}}F-F\right\Vert _{1/2}\nrightarrow 0,\,\,\,\text{%
as\thinspace \thinspace \thinspace }k\rightarrow \infty .  \label{kn3}
\end{equation}
\end{theorem}

\begin{theorem}
\label{theorem4}a) Let $0<p<1/2,$ $F\in H_{p}$, $\ \sup_{k}d\left(
n_{k}\right) =\infty $ and
\begin{equation}
\omega _{H_{p}}\left( 1/2^{\left\vert n_{k}\right\vert },F\right) =o\left(
1/2^{d\left( n_{k}\right) \left( 1/p-2\right) }\right) ,\text{ as \ }%
k\rightarrow \infty .  \label{cond3}
\end{equation}%
Then (\ref{fe2}) holds.

b) Let $\sup_{k}d\left( n_{k}\right) =\infty .$ Then there exists a
martingale $F\in H_{p}(G)$ $\left( 0<p<1/2\right) ,$\ \ for which
\begin{equation}
\omega _{H_{p}}\left( 1/2^{\left\vert n_{k}\right\vert },F\right) =O\left(
1/2^{d\left( n_{k}\right) \left( 1/p-2\right) }\right) ,\text{ \ as \ }%
k\rightarrow \infty  \label{cond4}
\end{equation}%
and
\begin{equation}
\left\Vert \sigma _{n_{k}}F-F\right\Vert _{L_{p,\infty }}\nrightarrow
0,\,\,\,\text{as\thinspace \thinspace \thinspace }k\rightarrow \infty .
\label{kn4}
\end{equation}
\end{theorem}

\section{AUXILIARY PROPOSITIONS}

\begin{lemma}[\textbf{Weisz \protect\cite{We3} (see also Simon \protect\cite%
{S})}]
\label{lemma0}A martingale $F=\left( F_{n},\text{ }n\in \mathbb{N}\right) $
is in $H_{p}\left( 0<p\leq 1\right) $ if and only if there exists a sequence
$\left( a_{k},k\in \mathbb{N}\right) $ of p-atoms and a sequence $\left( \mu
_{k},k\in \mathbb{N}\right) $ of a real numbers, such that for every $n\in
\mathbb{N}$
\end{lemma}

\begin{equation}
\qquad \sum_{k=0}^{\infty }\mu _{k}S_{2^{n}}a_{k}=F_{n},\text{ \ \ \ }%
\sum_{k=0}^{\infty }\left\vert \mu _{k}\right\vert ^{p}<\infty .  \label{6}
\end{equation}

Moreover, $\left\Vert F\right\Vert _{H_{p}}\backsim \inf \left(
\sum_{k=0}^{\infty }\left\vert \mu _{k}\right\vert ^{p}\right) ^{1/p}$,
where the infimum is taken over all decomposition of $F$ of the form (\ref{6}%
).

\begin{lemma}[\textbf{Weisz \protect\cite{We1}}]
\label{lemma1}Suppose that an operator $T$ is $\sigma $-linear and
\end{lemma}

\begin{equation*}
\int\limits_{\overset{-}{I}}\left\vert Ta\right\vert ^{p}d\mu \leq
c_{p}<\infty ,\text{ \ \ }\left( 0<p\leq 1\right)
\end{equation*}%
\textit{for every }$p$\textit{-atom }$a$\textit{,where }$I$\textit{\ denote
the support of the atom. If }$T$\textit{\ is bounded from }$L_{\infty \text{
}}$\textit{\ to }$L_{\infty },$\textit{\ then} \
\begin{equation*}
\left\Vert TF\right\Vert _{p}\leq c_{p}\left\Vert F\right\Vert _{H_{p}}.
\end{equation*}

\begin{lemma}[\textbf{see e.g. \protect\cite{G-E-S}, \protect\cite{sws}}]
\label{lemma2}Let $t,n\in \mathbb{N}.$ Then
\begin{equation*}
K_{2^{n}}\left( x\right) =\left\{
\begin{array}{c}
\text{ }2^{t-1},\text{\ if \ \ }x\in I_{n}\left( e_{t}\right) ,\text{ }n>t,%
\text{\ }x\in I_{t}\backslash I_{t+1}, \\
\left( 2^{n}+1\right) /2,\text{\ if \ \ }x\in I_{n}, \\
0,\text{\ otherwise.\ }%
\end{array}%
\right.
\end{equation*}
\end{lemma}

\begin{lemma}[\textbf{Tephnadze \protect\cite{tep7}}]
\label{lemma3}Let $n=\sum_{i=1}^{s}\sum_{k=l_{i}}^{m_{i}}2^{k},$ where $%
m_{1}\geq l_{1}>l_{1}-2\geq m_{2}\geq l_{2}>l_{2}-2>...>m_{s}\geq l_{s}\geq
0.$ Then
\begin{equation*}
n\left\vert K_{n}\left( x\right) \right\vert \geq 2^{2l_{i}-4},\text{ \ \
for \ \ }x\in E_{l_{i}}:=I_{l_{i}+1}\left( e_{l_{i}-1}+e_{l_{i}}\right) ,
\end{equation*}%
where $I_{1}\left( e_{-1}+e_{0}\right) =I_{2}\left( e_{0}+e_{1}\right) .$
\end{lemma}

\begin{lemma}[\textbf{Goginava \protect\cite{GoSzeged}}]
\label{lemma4}Let $x\in I_{M}^{k,l},$ $k=0,...,M-1,$ $l=k+1,...,M.$ Then
\end{lemma}

\begin{equation*}
\int_{I_{M}}\left\vert K_{n}\left( x+t\right) \right\vert d\mu \left(
t\right) \leq c2^{k+l-2M},\text{ for }n\geq 2^{M}.
\end{equation*}

\begin{lemma}
\label{lemma5}Let $n=\sum_{i=1}^{r}\sum_{k=l_{i}}^{m_{i}}2^{k},$ where $%
m_{1}\geq l_{1}>l_{1}-2\geq m_{2}\geq l_{2}>l_{2}-2>...>m_{s}\geq l_{s}\geq
0.$ Then
\begin{equation*}
\left\vert nK_{n}\right\vert \leq c\sum_{A=1}^{r}\left( 2^{l_{A}}\left\vert
K_{2^{l_{A}}}\right\vert +2^{m_{A}}\left\vert K_{2^{m_{A}}}\right\vert
+2^{l_{A}}\sum_{k=l_{A}}^{m_{A}}D_{2^{k}}\right) +cV\left( n\right) .
\end{equation*}
\end{lemma}

\begin{proof}
Let $n=\sum_{i=1}^{r}2^{n_{i}},$ $n_{1}>n_{2}>...>n_{r}\geq 0.$ By using (%
\ref{9a}) we get that%
\begin{equation*}
nK_{n}=\sum_{A=1}^{r}\left( \underset{j=1}{\overset{A-1}{\prod }}%
w_{2^{n_{j}}}\right) \left( \left( 2^{n_{A}}K_{2^{n_{A}}}+\left(
2^{n_{A}}-1\right) D_{2^{n_{A}}}\right) \right)
\end{equation*}%
\begin{equation*}
-\sum_{A=1}^{r}\left( \underset{j=1}{\overset{A-1}{\prod }}%
w_{2^{n_{j}}}\right) \left( 2^{n_{A}}-1-n^{\left( A\right) }\right)
D_{2^{n_{A}}}=I_{1}-I_{2}.
\end{equation*}

For $I_{1}$ we have that%
\begin{equation*}
I_{1}=\sum_{v=1}^{r}\left( \underset{j=1}{\overset{v-1}{\prod }}\underset{%
i=l_{j}}{\overset{m_{j}}{\prod }}w_{2^{i}}\right) \left(
\sum_{k=l_{v}}^{m_{v}}\left( \underset{j=k+1}{\overset{m_{v}}{\prod }}%
w_{2^{j}}\right) \left( 2^{k}K_{2^{k}}-\left( 2^{k}-1\right)
D_{2^{k}}\right) \right)
\end{equation*}%
\begin{equation*}
=\sum_{v=1}^{r}\left( \underset{j=1}{\overset{v-1}{\prod }}\underset{i=l_{j}}%
{\overset{m_{j}}{\prod }}w_{2^{i}}\right) \left(
\sum_{k=0}^{m_{v}}-\sum_{k=0}^{l_{v}-1}\right) \left( \underset{j=k+1}{%
\overset{m_{v}}{\prod }}w_{2^{j}}\right) \left( 2^{k}K_{2^{k}}-\left(
2^{k}-1\right) D_{2^{k}}\right)
\end{equation*}%
\begin{equation*}
=\sum_{v=1}^{r}\left( \underset{j=1}{\overset{v-1}{\prod }}\underset{i=l_{j}}%
{\overset{m_{j}}{\prod }}w_{2^{i}}\right) \left( \sum_{k=0}^{m_{v}}\left(
\underset{j=k+1}{\overset{m_{v}}{\prod }}w_{2^{j}}\right) \left(
2^{k}K_{2^{k}}-\left( 2^{k}-1\right) D_{2^{k}}\right) \right)
\end{equation*}%
\begin{equation*}
-\sum_{v=1}^{r}\left( \underset{j=1}{\overset{v}{\prod }}\underset{i=l_{j}}{%
\overset{m_{j}}{\prod }}w_{2^{i}}\right) \left( \sum_{k=0}^{l_{v}-1}\left(
\underset{j=k+1}{\overset{l_{v}-1}{\prod }}w_{2^{j}}\right) \left(
2^{k}K_{2^{k}}-\left( 2^{k}-1\right) D_{2^{k}}\right) \right) .
\end{equation*}

Since $2^{n}-1=\sum_{k=0}^{n-1}2^{k}$ and%
\begin{equation*}
\left( 2^{n}-1\right) K_{2^{n}-1}=\sum_{k=0}^{n-1}\left(
\prod_{j=k+1}^{n-1}w_{2^{j}}\right) \left( 2^{k}K_{2^{k}}-\left(
2^{k}-1\right) D_{2^{k}}\right) ,
\end{equation*}%
\ we obtain that
\begin{equation*}
I_{1}=\sum_{v=1}^{r}\left( \underset{j=1}{\overset{v-1}{\prod }}\underset{%
i=l_{j}}{\overset{m_{j}}{\prod }}w_{2^{i}}\right) \left(
2^{m_{v}+1}-1\right) K_{2^{m_{v}+1}-1}
\end{equation*}%
\begin{equation*}
-\sum_{v=1}^{r}\left( \underset{j=1}{\overset{v}{\prod }}\underset{i=l_{j}}{%
\overset{m_{j}}{\prod }}w_{2^{i}}\right) \left( 2^{l_{v}}-1\right)
K_{2^{l_{v}}-1}.
\end{equation*}%
By using $\left\vert K_{2^{n}}\right\vert \leq c\left\vert
K_{2^{n-1}}\right\vert ,\left\vert K_{2^{n}-1}\right\vert \leq c\left\vert
K_{2^{n}}\right\vert +c,\ $we have that%
\begin{equation}
\left\vert I_{1}\right\vert \leq c\sum_{v=1}^{r}\left( 2^{l_{v}}\left\vert
K_{2^{l_{v}}}\right\vert +2^{m_{v}}\left\vert K_{2^{m_{v}}}\right\vert
+cr\right) .  \label{12c}
\end{equation}

Let $l_{j}<n_{A}\leq m_{j},$ for some $j=1,...,s.$ Then $n^{\left( A\right)
}\geq \sum_{v=l_{j}}^{n_{A}-1}2^{v}\geq 2^{n_{A}}-2^{l_{j}}$ \ and $%
2^{n_{A}}-1-n^{\left( A\right) }\leq 2^{l_{j}}.$ If $l_{j}=n_{A}$ for some $%
j=1,...,s.$ Then $n^{\left( A\right) }\leq 2^{m_{j-1}+1}<2^{l_{j}}.$ Hence
\begin{equation}
\left\vert I_{2}\right\vert \leq
c\sum_{v=1}^{r}2^{l_{v}}\sum_{k=l_{v}}^{m_{v}}D_{2^{k}}.  \label{12d}
\end{equation}

By combining (\ref{12c})-(\ref{12d}) we complete the proof of Lemma \ref%
{lemma5}.
\end{proof}

\section{Proof of the Theorems}

\begin{proof}[\textbf{Proof of Theorem \protect\ref{theorem1}.}]
Suppose that
\begin{equation}
\left\Vert \frac{\sigma _{n}F}{V^{2}\left( n\right) }\right\Vert _{1/2}\leq
c\left\Vert F\right\Vert _{H_{1/2}}.  \label{12k}
\end{equation}%
By combining (\ref{5.1}) and (\ref{12k}) we have that
\begin{equation}
\left\Vert \frac{\sigma _{n}F}{V^{2}\left( n\right) }\right\Vert
_{H_{1/2}}\sim \int_{G}\left\Vert \frac{\widetilde{\left( \sigma
_{n}F\right) ^{\left( t\right) }}}{V^{2}\left( n\right) }\right\Vert
_{1/2}d\mu \left( t\right)  \label{12l}
\end{equation}%
\begin{equation*}
=\int_{G}\left\Vert \frac{\sigma _{n}\widetilde{F^{\left( t\right) }}}{%
V^{2}\left( n\right) }\right\Vert _{1/2}d\mu \left( t\right) \leq
c\int_{G}\left\Vert \widetilde{F^{\left( t\right) }}\right\Vert
_{H_{1/2}}d\mu \left( t\right)
\end{equation*}%
\begin{equation*}
\leq c\int_{G}\left\Vert F\right\Vert _{H_{1/2}}d\mu \left( t\right)
=c\left\Vert F\right\Vert _{H_{1/2}}.
\end{equation*}

By Lemma \ref{lemma1} and (\ref{12l}), the proof of theorem \ref{theorem1}
will be complete, if we show that%
\begin{equation*}
\int_{\overline{I_{M}}}\left( \frac{\left\vert \sigma _{n}a\left( x\right)
\right\vert }{V^{2}\left( n\right) }\right) ^{1/2}d\mu \left( x\right) \leq
c<\infty ,
\end{equation*}%
for every 1/2-atom $a.$ We may assume that $a$ be an arbitrary $1/2$-atom,
with support$\ I,$ $\mu \left( I\right) =2^{-M}$ and $I=I_{M}.$ It is easy
to see that $\sigma _{n}\left( a\right) =0,$ when $n\leq 2^{M}.$ Therefore,
we can suppose that $n>2^{M}.$

Set%
\begin{eqnarray*}
II_{\alpha _{A}}^{1}\left( x\right) &:&=2^{M}\int_{I_{M}}2^{\alpha
_{A}}\left\vert K_{2^{\alpha _{A}}}\left( x+t\right) \right\vert d\mu \left(
t\right) ,\text{\ } \\
II_{l_{A}}^{2}\left( x\right)
&:&=2^{M}\int_{I_{M}}2^{l_{A}}\sum_{k=l_{A}}^{m_{A}}D_{2^{k}}\left(
x+t\right) d\mu \left( t\right) ,
\end{eqnarray*}

Let $x\in I_{M}.$ Since $\sigma _{n}$ is bounded from $L_{\infty }$ to $%
L_{\infty },$ $n>2^{M}$ and $\left\Vert a\right\Vert _{\infty }\leq 2^{2M},$
by using Lemma \ref{lemma5} we obtain that
\begin{equation*}
\frac{\left\vert \sigma _{n}a\left( x\right) \right\vert }{V^{2}\left(
n\right) }\leq \frac{c}{V^{2}\left( n\right) }\int_{I_{M}}\left\vert a\left(
x\right) \right\vert \left\vert K_{n}\left( x+t\right) \right\vert d\mu
\left( t\right)
\end{equation*}%
\begin{equation*}
\leq \frac{\left\Vert a\right\Vert _{\infty }}{V^{2}\left( n\right) }%
\int_{I_{M}}\left\vert K_{n}\left( x+t\right) \right\vert d\mu \left(
t\right) \leq \frac{c2^{2M}}{V^{2}\left( n\right) }\int_{I_{M}}\left\vert
K_{n}\left( x+t\right) \right\vert d\mu \left( t\right)
\end{equation*}%
\begin{equation*}
\leq \frac{c2^{M}}{V^{2}\left( n\right) }\sum_{A=1}^{s}\int_{I_{M}}2^{l_{A}}%
\left\vert K_{2^{l_{A}}}\left( x+t\right) \right\vert d\mu \left( t\right)
\end{equation*}%
\begin{equation*}
+\frac{c2^{M}}{V^{2}\left( n\right) }\int_{I_{M}}2^{m_{A}}\left\vert
K_{2^{m_{A}}}\left( x+t\right) \right\vert d\mu \left( t\right)
\end{equation*}%
\begin{equation*}
+\frac{c2^{M}}{V^{2}\left( n\right) }\sum_{A=1}^{s}\int_{I_{M}}2^{l_{A}}%
\sum_{k=l_{A}}^{m_{A}}D_{2^{k}}\left( x+t\right) d\mu \left( t\right) +\frac{%
c2^{M}}{V^{2}\left( n\right) }\int_{I_{M}}V\left( n\right) d\mu \left(
t\right) .
\end{equation*}%
\begin{equation*}
\frac{c}{V^{2}\left( n\right) }\sum_{A=1}^{s}\left( II_{^{l_{A}}}^{1}\left(
x\right) +II_{^{m_{A}}}^{1}\left( x\right) +II_{l_{A}}^{2}\left( x\right)
\right) +c.
\end{equation*}

Hence
\begin{equation*}
\int_{\overline{I_{M}}}\left\vert \frac{\sigma _{n}a\left( x\right) }{%
V^{2}\left( n\right) }\right\vert ^{1/2}d\mu \left( x\right)
\end{equation*}%
\begin{eqnarray*}
&\leq &\frac{c}{V\left( n\right) }\left( \sum_{A=1}^{s}\int_{\overline{I_{M}}%
}\left\vert II_{l_{A}}^{1}\left( x\right) \right\vert ^{1/2}d\mu \left(
x\right) \right. \\
&&\left. +\int_{\overline{I_{M}}}\left\vert II_{m_{A}}^{1}\left( x\right)
\right\vert ^{1/2}d\mu \left( x\right) +\int_{\overline{I_{M}}}\left\vert
II_{l_{A}}^{2}\left( x\right) \right\vert ^{1/2}d\mu \left( x\right) \right)
+c
\end{eqnarray*}

Since $s\leq 4V\left( n\right) ,$ we obtain that Theorem \ref{theorem1} will
be proved if we show that
\begin{equation}
\int_{\overline{I_{M}}}\left\vert II_{\alpha _{A}}^{1}\left( x\right)
\right\vert ^{1/2}d\mu \left( x\right) \leq c<\infty ,\text{\ }\int_{%
\overline{I_{M}}}\left\vert II_{l_{A}}^{2}\left( x\right) \right\vert
^{1/2}d\mu \left( x\right) \leq c<\infty ,\text{\ }  \label{11.1}
\end{equation}%
where $\alpha _{A}=l_{A}$ or $\alpha _{A}=m_{A},$ $A=1,...,s$.

Let $t\in I_{M}$ and $x\in I_{l+1}\left( e_{k}+e_{l}\right) ,$ $0\leq
k<l<\alpha _{A}\leq M$ or $0\leq k<l\leq M\leq \alpha _{A}.$ Since $x+t\in
I_{l+1}\left( e_{k}+e_{l}\right) ,$ by using Lemma \ref{lemma2} we obtain
that
\begin{equation}
K_{2^{\alpha _{A}}}\left( x+t\right) =0\text{ \ and \ \ }II_{\alpha
_{A}}^{1}\left( x\right) =0.  \label{10a}
\end{equation}

Let $x\in I_{l+1}\left( e_{k}+e_{l}\right) ,$ $0\leq k<\alpha _{A}\leq l\leq
M.$ Then $x+t\in I_{l+1}\left( e_{k}+e_{l}\right) ,$ for $t\in I_{M}$ and if
we apply lemma \ref{lemma2} we get that
\begin{equation}
2^{\alpha _{A}}\left\vert K_{2^{\alpha _{A}}}\left( x+t\right) \right\vert
\leq 2^{\alpha _{A}+k}\text{ \ and\ \ \ \ }II_{\alpha _{A}}^{1}\left(
x\right) \leq 2^{\alpha _{A}+k}.  \label{10b}
\end{equation}

Analogously to (\ref{10b}) we can show that if $0\leq \alpha _{A}\leq
k<l\leq M$, then
\begin{equation}
2^{\alpha _{A}}\left\vert K_{2^{\alpha _{A}}}\left( x+t\right) \right\vert
\leq 2^{2\alpha _{A}},\text{ \ }II_{\alpha _{A}}^{1}\left( x\right) \leq
2^{2\alpha _{A}},\text{\ }t\in I_{M},\text{\ }x\in I_{l+1}\left(
e_{k}+e_{l}\right) ,  \label{10c}
\end{equation}

Let $0\leq \alpha _{A}\leq M-1.$ By combining (\ref{1}) and (\ref{10a}-\ref%
{10c}) we have that
\begin{equation*}
\int_{\overline{I_{M}}}\left\vert II_{\alpha _{A}}^{1}\left( x\right)
\right\vert ^{1/2}d\mu \left( x\right)
\end{equation*}%
\begin{equation*}
=\overset{M-2}{\underset{k=0}{\sum }}\overset{M-1}{\underset{l=k+1}{\sum }}%
\int_{I_{l+1}\left( e_{k}+e_{l}\right) }\left\vert II_{\alpha
_{A}}^{1}\left( x\right) \right\vert ^{1/2}d\mu \left( x\right) +\overset{M-1%
}{\underset{k=0}{\sum }}\int_{I_{M}\left( e_{k}\right) }\left\vert
II_{\alpha _{A}}^{1}\left( x\right) \right\vert ^{1/2}d\mu \left( x\right)
\end{equation*}%
\begin{equation*}
\leq c\overset{\alpha _{A}-1}{\underset{k=0}{\sum }}\overset{M-1}{\underset{%
l=\alpha _{A}+1}{\sum }}\int_{I_{l+1}\left( e_{k}+e_{l}\right) }2^{\left(
\alpha _{A}+k\right) /2}d\mu \left( x\right) +c\overset{M-2}{\underset{%
k=\alpha _{A}}{\sum }}\overset{M-1}{\underset{l=k+1}{\sum }}%
\int_{I_{l+1}\left( e_{k}+e_{l}\right) }2^{\alpha _{A}}d\mu \left( x\right)
\end{equation*}%
\begin{equation*}
+c\overset{\alpha _{A}-1}{\underset{k=0}{\sum }}\int_{I_{M}\left(
e_{k}\right) }2^{\left( \alpha _{A}+k\right) /2}d\mu \left( x\right) +c%
\overset{M-1}{\underset{k=\alpha _{A}}{\sum }}\int_{I_{M}\left( e_{k}\right)
}2^{\alpha _{A}}d\mu \left( x\right)
\end{equation*}%
\begin{equation*}
\leq c\overset{\alpha _{A}-1}{\underset{k=0}{\sum }}\overset{M-1}{\underset{%
l=\alpha _{A}+1}{\sum }}\frac{2^{\left( \alpha _{A}+k\right) /2}}{2^{l}}+c%
\overset{M-2}{\underset{k=\alpha _{A}}{\sum }}\overset{M-1}{\underset{l=k+1}{%
\sum }}\frac{2^{\alpha _{A}}}{2^{l}}+c\overset{\alpha _{A}-1}{\underset{k=0}{%
\sum }}\frac{2^{\left( \alpha _{A}+k\right) /2}}{2^{M}}+c\overset{M-1}{%
\underset{k=\alpha _{A}}{\sum }}\frac{2^{\alpha _{A}}}{2^{M}}
\end{equation*}%
\begin{equation*}
\leq c<\infty ,\text{ \ \ for any }A=1,...,s.
\end{equation*}

Let $\alpha _{A}\geq M.$ Analogously we can show that (\ref{11.1}) holds,
for $II_{\alpha _{A}}^{1}\left( x\right) ,$ $A=1,...,s.$

Now, prove boundedness of $II_{l_{A}}^{2}$. Let $t\in I_{M}$ and $x\in
I_{i}\backslash I_{i+1},$ $i\leq l_{A}-1.$ Since $x+t\in I_{i}\backslash
I_{i+1},$ by using (\ref{1dn}) we have that
\begin{equation}
II_{l_{A}}^{2}\left( x\right) =0.  \label{13a}
\end{equation}

Let $x\in I_{i}\backslash I_{i+1},$ $l_{A}\leq i\leq m_{A}.$ Since $n\geq
2^{M}$ and $t\in I_{M},$ by using (\ref{1dn}) we obtain that
\begin{equation}
II_{l_{A}}^{2}\left( x\right) \leq
2^{M}\int_{I_{M}}2^{l_{A}}\sum_{k=l_{A}}^{i}D_{2^{k}}\left( x+t\right) d\mu
\left( t\right) \leq c2^{l_{A}+i}.  \label{13b}
\end{equation}

Let $x\in I_{i}\backslash I_{i+1},$ $m_{A}<i\leq M-1.$ By using (\ref{1dn})
we get that $x+t\in I_{i}\backslash I_{i+1},$ for $t\in I_{M}$ and
\begin{equation}
II_{l_{A}}^{2}\left( x\right) \leq c2^{M}\int_{I_{M}}2^{l_{A}+m_{A}}\leq
c2^{l_{A}+m_{A}}.  \label{13c}
\end{equation}%
Let $0\leq l_{A}\leq m_{A}\leq M.$ By combining (\ref{1}) and (\ref{13a}-\ref%
{13c}) we get that
\begin{equation*}
\int_{\overline{I_{M}}}\left\vert II_{l_{A}}^{2}\left( x\right) \right\vert
^{1/2}d\mu \left( x\right) =\left(
\sum_{i=0}^{l_{A}-1}+\sum_{i=l_{A}}^{m_{A}}+\sum_{i=m_{A}+1}^{M-1}\right)
\int_{I_{i}\backslash I_{i+1}}\left\vert II_{l_{A}}^{2}\left( x\right)
\right\vert ^{1/2}d\mu \left( x\right)
\end{equation*}%
\begin{equation*}
\leq c\sum_{i=l_{A}}^{m_{A}}\int_{I_{i}\backslash I_{i+1}}2^{\left(
l_{A}+i\right) /2}d\mu \left( x\right)
+c\sum_{i=m_{A}+1}^{M-1}\int_{I_{i}\backslash I_{i+1}}2^{\left(
l_{A}+m_{A}\right) /2}d\mu \left( x\right)
\end{equation*}%
\begin{equation*}
\leq c\sum_{i=l_{A}}^{m_{A}}2^{\left( l_{A}+i\right) /2}\frac{1}{2^{i}}%
+c\sum_{i=m_{A}+1}^{M-1}2^{\left( l_{A}+m_{A}\right) /2}\frac{1}{2^{i}}\leq
c<\infty .
\end{equation*}

Analogously we can prove, when $0\leq l_{A}\leq M<m_{A}$ or $M\leq l_{A}\leq
m_{A}.$

Now, prove second part of Theorem \ref{theorem1}. Under condition (\ref{30}%
), there exists an increasing sequence $\left\{ \alpha _{k}:\text{ }k\geq
0\right\} \subset \left\{ n_{k}\text{ }:k\geq 0\right\} $ of the positive
integers, such that
\begin{equation}
\sum_{k=1}^{\infty }\Phi ^{1/4}\left( \alpha _{k}\right) /V^{1/2}\left(
\alpha _{k}\right) \leq c<\infty .  \label{2aaa}
\end{equation}

Let \qquad
\begin{equation*}
F_{A}:=\sum_{\left\{ k;\text{ }\left\vert \alpha _{k}\right\vert <A\right\}
}\lambda _{k}a_{k},\text{\ }
\end{equation*}%
\begin{equation*}
\lambda _{k}:=\Phi ^{1/2}\left( \alpha _{k}\right) /V\left( \alpha
_{k}\right) ,\text{\ }a_{k}:=2^{\left\vert \alpha _{k}\right\vert }\left(
D_{2^{\left\vert \alpha _{k}\right\vert +1}}-D_{2^{\left\vert \alpha
_{k}\right\vert }}\right) .
\end{equation*}

Since$\ $%
\begin{equation}
S_{2^{A}}a_{k}=\left\{
\begin{array}{ll}
a_{k} & \left\vert \alpha _{k}\right\vert <A, \\
0\, & \left\vert \alpha _{k}\right\vert \geq A,%
\end{array}%
\right.  \label{4aa}
\end{equation}%
\begin{equation*}
\text{supp}(a_{k})=I_{\left\vert \alpha _{k}\right\vert },\text{\ }%
\int_{I\left\vert \alpha _{k}\right\vert }a_{k}d\mu =0,\text{\ }\left\Vert
a_{k}\right\Vert _{\infty }\leq 2^{2\left\vert \alpha _{k}\right\vert }=\mu (%
\text{supp }a_{k})^{-2},
\end{equation*}%
if we apply Lemma \ref{lemma0} and (\ref{2aaa}) we conclude that $F=\left(
F_{1},F_{2},...\right) \in H_{1/2}.$

It is easy to show that%
\begin{equation}
\widehat{F}(j)  \label{5aa}
\end{equation}%
\begin{equation*}
=\left\{
\begin{array}{ll}
2^{\left\vert \alpha _{k}\right\vert }\Phi ^{1/2}\left( \alpha _{k}\right)
/V\left( \alpha _{k}\right) , & \text{\thinspace \thinspace }j\in \left\{
2^{\left\vert \alpha _{k}\right\vert },...,2^{_{\left\vert \alpha
_{k}\right\vert +1}}-1\right\} ,\text{ }k=0,1,... \\
0\,, & \text{\thinspace }j\notin \bigcup\limits_{k=0}^{\infty }\left\{
2^{_{\left\vert \alpha _{k}\right\vert }},...,2^{_{\left\vert \alpha
_{k}\right\vert +1}}-1\right\} .\text{ }%
\end{array}%
\right.
\end{equation*}

Let $2^{\left\vert \alpha _{k}\right\vert }<j<\alpha _{k}.$ By using (\ref%
{5aa}) we get that%
\begin{equation}
S_{j}F=S_{2^{\left\vert \alpha _{k}\right\vert }}F+\sum_{v=2^{^{\left\vert
\alpha _{k}\right\vert }}}^{j-1}\widehat{F}(v)w_{v}=S_{2^{\left\vert \alpha
_{k}\right\vert }}F+\frac{\left( D_{_{j}}-D_{2^{\left\vert \alpha
_{k}\right\vert }}\right) 2^{\left\vert \alpha _{k}\right\vert }\Phi
^{1/2}\left( \alpha _{k}\right) }{V\left( \alpha _{k}\right) }  \label{sn}
\end{equation}

Hence%
\begin{equation}
\frac{\sigma _{_{\alpha _{k}}}F}{\Phi \left( \alpha _{k}\right) }=\frac{1}{%
\Phi \left( \alpha _{k}\right) \alpha _{k}}\sum_{j=1}^{2^{\left\vert \alpha
_{k}\right\vert }}S_{j}F+\frac{1}{\Phi \left( \alpha _{k}\right) \alpha _{k}}%
\sum_{j=2^{\left\vert \alpha _{k}\right\vert }+1}^{\alpha _{k}}S_{j}F
\label{7aaa}
\end{equation}%
\begin{equation*}
=\frac{\sigma _{_{2^{\left\vert \alpha _{k}\right\vert }}}F}{\Phi \left(
\alpha _{k}\right) \alpha _{k}}+\frac{\left( \alpha _{k}-2^{\left\vert
\alpha _{k}\right\vert }\right) S_{2^{\left\vert \alpha _{k}\right\vert }}F}{%
\Phi \left( \alpha _{k}\right) \alpha _{k}}+\frac{2^{\left\vert \alpha
_{k}\right\vert }\Phi ^{1/2}\left( \alpha _{k}\right) }{\Phi \left( \alpha
_{k}\right) V\left( \alpha _{k}\right) \alpha _{k}}\sum_{j=2^{_{\left\vert
\alpha _{k}\right\vert }}+1}^{\alpha _{k}}\left( D_{_{j}}-D_{2^{\left\vert
\alpha _{k}\right\vert }}\right)
\end{equation*}%
\begin{equation*}
=III_{1}+III_{2}+III_{3}.
\end{equation*}%
Since%
\begin{equation}
D_{j+2^{m}}=D_{2^{m}}+w_{_{2^{m}}}D_{j},\text{ \qquad when \thinspace
\thinspace }j<2^{m}  \label{f1}
\end{equation}%
we obtain that%
\begin{equation}
\left\vert III_{3}\right\vert =\frac{2^{\left\vert \alpha _{k}\right\vert
}\Phi ^{1/2}\left( \alpha _{k}\right) }{\Phi \left( \alpha _{k}\right)
V\left( \alpha _{k}\right) \alpha _{k}}\left\vert \sum_{j=1}^{\alpha
_{k}-2^{_{\left\vert \alpha _{k}\right\vert }}}\left( D_{j+2^{_{\left\vert
\alpha _{k}\right\vert }}}-D_{2^{\left\vert \alpha _{k}\right\vert }}\right)
\right\vert  \label{9aaa}
\end{equation}%
\begin{equation*}
=\frac{2^{\left\vert \alpha _{k}\right\vert }\Phi ^{1/2}\left( \alpha
_{k}\right) }{\Phi \left( \alpha _{k}\right) V\left( \alpha _{k}\right)
\alpha _{k}}\left\vert \sum_{j=1}^{\alpha _{k}-2^{_{\left\vert \alpha
_{k}\right\vert }}}D_{j}\right\vert =\frac{2^{\left\vert \alpha
_{k}\right\vert }\Phi ^{1/2}\left( \alpha _{k}\right) }{\Phi \left( \alpha
_{k}\right) V\left( \alpha _{k}\right) \alpha _{k}}\left( \alpha
_{k}-2^{_{\left\vert \alpha _{k}\right\vert }}\right) \left\vert K_{\alpha
_{k}-2^{_{\left\vert \alpha _{k}\right\vert }}}\right\vert
\end{equation*}%
\begin{equation*}
\geq c\left( \alpha _{k}-2^{_{\left\vert \alpha _{k}\right\vert }}\right)
\left\vert K_{\alpha _{k}-2^{_{\left\vert \alpha _{k}\right\vert
}}}\right\vert /\left( \Phi ^{1/2}\left( \alpha _{k}\right) V\left( \alpha
_{k}\right) \right) .
\end{equation*}

\textbf{\ }Let $\alpha
_{k}=\sum_{i=1}^{r_{k}}\sum_{k=l_{i}^{k}}^{m_{i}^{k}}2^{k},$ \ where $%
m_{1}^{k}\geq l_{1}^{k}>l_{1}^{k}-2\geq m_{2}^{k}\geq
l_{2}^{k}>l_{2}^{k}-2\geq ...\geq m_{s}^{k}\geq l_{s}^{k}\geq 0.$ Since (see
Theorem \ref{theorem1} and (\ref{markW1})) $\left\Vert III_{1}\right\Vert
_{1/2}\leq c,$\ $\left\Vert III_{2}\right\Vert _{1/2}\leq c,$ $\mu \left\{
E_{l_{i}^{k}}\right\} \geq 1/2^{l_{i}^{k}-1},$ by combining (\ref{7aaa}), (%
\ref{9aaa}) and Lemma \ref{lemma3} we obtain that
\begin{equation*}
\int_{G}\left\vert \sigma _{\alpha _{k}}F(x)/\Phi \left( \alpha _{k}\right)
\right\vert ^{1/2}d\mu \left( x\right) \geq \left\Vert III_{3}\right\Vert
_{1/2}^{1/2}-\left\Vert III_{2}\right\Vert _{1/2}^{1/2}-\left\Vert
III_{1}\right\Vert _{1/2}^{1/2}
\end{equation*}%
\begin{equation*}
\geq c\text{ }\underset{i=2}{\overset{s_{k}-2}{\sum }}\int_{E_{l_{i}^{k}}}%
\left\vert 2^{2l_{i}^{k}}/\left( \Phi ^{1/2}\left( \alpha _{k}\right)
V\left( \alpha _{k}\right) \right) \right\vert ^{1/2}d\mu \left( x\right)
\end{equation*}%
\begin{equation*}
\geq c\overset{s_{k}-2}{\underset{i=2}{\sum }}1/\left( V^{1/2}\left( \alpha
_{k}\right) \Phi ^{1/4}\left( \alpha _{k}\right) \right) \geq cs_{k}/\left(
V^{1/2}\left( \alpha _{k}\right) \Phi ^{1/4}\left( \alpha _{k}\right) \right)
\end{equation*}%
\begin{equation*}
\geq cV^{1/2}\left( \alpha _{k}\right) /\Phi ^{1/4}\left( \alpha _{k}\right)
\rightarrow \infty ,\text{ as }k\rightarrow \infty .
\end{equation*}

Theorem \ref{theorem1} is proved.
\end{proof}

\begin{proof}[\textbf{Proof of Theorem \protect\ref{theorem2}.}]
Let $n\in \mathbb{N}.$ After analogously steps of (\ref{12l}) we only have
show that%
\begin{equation*}
\int\limits_{\overline{I}_{M}}\left( 2^{d\left( n\right) \left( 2-1/p\right)
}\left\vert \sigma _{n}\left( a\right) \right\vert \right) ^{p}d\mu \leq
c_{p}<\infty ,
\end{equation*}%
for every p-atom $a,$ where $I$ denotes the support of the atom.

Analogously of first part of Theorem \ref{theorem1}, we may assume that $a$
be an arbitrary p-atom, with support$\ I=I_{M}$, $\mu \left( I_{M}\right)
=2^{-M}$ and $n>2^{M}.$ Since $\left\Vert a\right\Vert _{\infty }\leq
2^{M/p} $ we can write that
\begin{equation*}
2^{d\left( n\right) \left( 2-1/p\right) }\left\vert \sigma _{n}\left(
a\right) \right\vert \leq 2^{d\left( n\right) \left( 2-1/p\right)
}\left\Vert a\right\Vert _{\infty }\int_{I_{M}}\left\vert K_{n}\left(
x+t\right) \right\vert d\mu \left( t\right)
\end{equation*}%
\begin{equation*}
\leq 2^{d\left( n\right) \left( 2-1/p\right) }2^{M/p}\int_{I_{M}}\left\vert
K_{n}\left( x+t\right) \right\vert d\mu \left( t\right) .
\end{equation*}

Let $x\in I_{l+1}\left( e_{k}+e_{l}\right) ,\,0\leq k,l\leq \left[ n\right]
\leq M.$ Then from Lemma \ref{lemma2} we get that $K_{n}\left( x+t\right)
=0, $ for $t\in I_{M}$ and
\begin{equation}
2^{d\left( n\right) \left( 2-1/p\right) }\left\vert \sigma _{n}\left(
a\right) \right\vert =0.  \label{12a}
\end{equation}

Let $x\in I_{l+1}\left( e_{k}+e_{l}\right) ,\,\left[ n\right] \leq k,l\leq M$
or $k\leq \left[ n\right] \leq l\leq M.$ By using Lemma \ref{lemma4} we can
write that
\begin{equation}
2^{d\left( n\right) \left( 2-1/p\right) }\left\vert \sigma _{n}\left(
a\right) \right\vert \leq 2^{d\left( n\right) \left( 2-1/p\right)
}2^{M\left( 1/p-2\right) +k+l}  \label{12}
\end{equation}%
\begin{equation*}
\leq c_{p}2^{\left[ n\right] \left( 1/p-2\right) +k+l}.
\end{equation*}

By combining (\ref{1}), (\ref{12a}) and (\ref{12}) we obtain that%
\begin{equation*}
\int_{\overline{I_{M}}}\left\vert 2^{d\left( n\right) \left( 2-1/p\right)
}\sigma _{n}a\left( x\right) \right\vert ^{p}d\mu \left( x\right)
\end{equation*}%
\begin{equation*}
\leq \left( \overset{\left[ n\right] -2}{\underset{k=0}{\sum }}\overset{%
\left[ n\right] -1}{\underset{l=k+1}{\sum }}+\overset{\left[ n\right] -1}{%
\underset{k=0}{\sum }}\overset{M-1}{\underset{l=\left[ n\right] }{\sum }}+%
\overset{M-2}{\underset{k=\left[ n\right] }{\sum }}\overset{M-1}{\underset{%
l=k+1}{\sum }}\right) \int_{I_{l+1}\left( e_{k}+e_{l}\right) }\left\vert
2^{d\left( n\right) \left( 2-1/p\right) }\sigma _{n}a\left( x\right)
\right\vert ^{p}d\mu \left( x\right)
\end{equation*}%
\begin{equation*}
+\overset{M-1}{\underset{k=0}{\sum }}\int_{I_{M}\left( e_{k}\right)
}\left\vert 2^{d\left( n\right) \left( 2-1/p\right) }\sigma _{n}a\left(
x\right) \right\vert ^{p}d\mu \left( x\right) \leq c_{p}\overset{M-2}{%
\underset{k=\left[ n\right] }{\sum }}\overset{M-1}{\underset{l=k+1}{\sum }}%
\frac{1}{2^{l}}2^{\left[ n\right] \left( 2p-1\right) }2^{p\left( k+l\right) }
\end{equation*}%
\begin{equation*}
+c_{p}\overset{\left[ n\right] }{\underset{k=0}{\sum }}\overset{M-1}{%
\underset{l=\left[ n\right] +1}{\sum }}\frac{1}{2^{l}}2^{\left[ n\right]
\left( 2p-1\right) }2^{p\left( k+l\right) }+\frac{c_{p}2^{\left[ n\right]
\left( 2p-1\right) }}{2^{M}}\overset{\left[ n\right] }{\underset{k=0}{\sum }}%
2^{p\left( k+M\right) }<c_{p}<\infty .
\end{equation*}

Now, prove second part of Theorem \ref{theorem2}. \textbf{\ }Under
conditions (\ref{31aaa}), there exists sequence $\left\{ \alpha _{k}:\text{ }%
k\geq 0\right\} \subset \left\{ n_{k}:\text{ }k\geq 0\right\} ,$ such that $%
\alpha _{0}\geq 3$ and
\begin{equation}
\sum_{\eta =0}^{\infty }u^{-p}\left( \alpha _{\eta }\right) <c_{p}<\infty ,%
\text{\ }u\left( \alpha _{k}\right) =2^{d\left( \alpha _{k}\right) \left(
1/p-2\right) /2}/\Phi ^{1/2}\left( \alpha _{k}\right) .  \label{121}
\end{equation}

Let \qquad
\begin{equation*}
F_{A}=\sum_{\left\{ k;\text{ }\left\vert \alpha _{k}\right\vert <A\right\}
}a_{k}^{p}/u\left( \alpha _{k}\right) ,\text{ \ }a_{k}^{\left( p\right)
}=2^{\left\vert \alpha _{k}\right\vert \left( 1/p-1\right) }\left(
D_{2^{\left\vert \alpha _{k}\right\vert +1}}-D_{2^{\left\vert \alpha
_{k}\right\vert }}\right) .
\end{equation*}

If we apply Lemma \ref{lemma0} and (\ref{121}), analogously to second part
of Theorem \ref{theorem2}, we conclude that $F\in H_{p}.$

It is easy to show that%
\begin{equation}
\widehat{F}(j)  \label{6aa}
\end{equation}%
\begin{equation*}
=\left\{
\begin{array}{ll}
2^{\left\vert \alpha _{k}\right\vert \left( 1/p-1\right) }/u\left( \alpha
_{k}\right) , & j\in \left\{ 2^{\left\vert \alpha _{k}\right\vert
},...,2^{_{\left\vert \alpha _{k}\right\vert +1}}-1\right\} ,\text{ }%
k=0,1,... \\
0\,, & j\notin \bigcup\limits_{k=0}^{\infty }\left\{ 2^{\left\vert \alpha
_{k}\right\vert },...,2^{_{\left\vert \alpha _{k}\right\vert +1}}-1\right\} .%
\text{ }%
\end{array}%
\right.
\end{equation*}

Let $2^{\left\vert \alpha _{k}\right\vert }<j<\alpha _{k}.$ By applying (\ref%
{6aa}), analogously to (\ref{sn}) and (\ref{7aaa}) we get that%
\begin{equation*}
\frac{\sigma _{_{\alpha _{k}}}F}{\Phi \left( \alpha _{k}\right) }=\frac{%
\sigma _{_{2^{\left\vert \alpha _{k}\right\vert }}}F}{\Phi \left( \alpha
_{k}\right) \alpha _{k}}+\frac{\left( \alpha _{k}-2^{\left\vert \alpha
_{k}\right\vert }\right) S_{2^{\left\vert \alpha _{k}\right\vert }}F}{\Phi
\left( \alpha _{k}\right) \alpha _{k}}
\end{equation*}%
\begin{equation*}
+\frac{2^{\left\vert \alpha _{k}\right\vert \left( 1/p-1\right) }}{\Phi
\left( \alpha _{k}\right) u\left( \alpha _{k}\right) \alpha _{k}}%
\sum_{j=2^{_{\left\vert \alpha _{k}\right\vert }}}^{\alpha _{k}-1}\left(
D_{_{j}}-D_{2^{\left\vert \alpha _{k}\right\vert }}\right)
=IV_{1}+IV_{2}+IV_{3}.
\end{equation*}

Let $\alpha _{k}\in \mathbb{N}$ and $E_{\left[ \alpha _{k}\right] }:=I_{_{%
\left[ \alpha _{k}\right] +1}}\left( e_{\left[ \alpha _{k}\right] -1}+e_{%
\left[ \alpha _{k}\right] }\right) .$ Since $\left[ \alpha
_{k}-2^{\left\vert \alpha _{k}\right\vert }\right] =\left[ \alpha _{k}\right]
,$ By combining (\ref{9a}), (\ref{f1}) and Lemma \ref{lemma2}, analogously
to (\ref{9aaa}), for $IV_{3}$ we have that
\begin{equation*}
\left\vert IV_{3}\right\vert =\frac{2^{\left\vert \alpha _{k}\right\vert
\left( 1/p-1\right) }}{\Phi \left( \alpha _{k}\right) u\left( \alpha
_{k}\right) \alpha _{k}}\left( \alpha _{k}-2^{\left\vert \alpha
_{k}\right\vert }\right) \left\vert K_{\alpha _{k}-2^{\left\vert \alpha
_{k}\right\vert }}\right\vert
\end{equation*}%
\begin{equation*}
=\frac{2^{\left\vert \alpha _{k}\right\vert \left( 1/p-1\right) }}{\Phi
\left( \alpha _{k}\right) u\left( \alpha _{k}\right) \alpha _{k}}\left\vert
2^{\left[ \alpha _{k}\right] }K_{\left[ \alpha _{k}\right] }\right\vert \geq
\frac{2^{\left\vert \alpha _{k}\right\vert \left( 1/p-2\right) }2^{2\left[
\alpha _{k}\right] -4}}{\Phi \left( \alpha _{k}\right) u\left( \alpha
_{k}\right) }
\end{equation*}%
\begin{equation*}
\geq 2^{\left\vert \alpha _{k}\right\vert \left( 1/p-2\right) /2}2^{2\left[
\alpha _{k}\right] -4}/\Phi ^{1/2}\left( \alpha _{k}\right) .
\end{equation*}

It follows that%
\begin{equation*}
\left\Vert IV_{3}\right\Vert _{L_{p,\infty }}^{p}
\end{equation*}%
\begin{equation*}
\geq \left( \frac{2^{\left\vert \alpha _{k}\right\vert \left( 1/p-2\right)
/2}2^{2\left[ \alpha _{k}\right] -4}}{\Phi ^{1/2}\left( \alpha _{k}\right) }%
\right) ^{p}\mu \left\{ x\in G:\text{ }\left\vert IV_{3}\right\vert \geq
\frac{2^{\left\vert \alpha _{k}\right\vert \left( 1/p-2\right) /2}2^{2\left[
\alpha _{k}\right] -4}}{\Phi ^{1/2}\left( \alpha _{k}\right) }\right\}
\end{equation*}%
\begin{equation*}
\geq c_{p}\left( 2^{2\left[ \alpha _{k}\right] +\left\vert \alpha
_{k}\right\vert \left( 1/p-2\right) /2}/\Phi ^{1/2}\left( \alpha _{k}\right)
\right) ^{p}\mu \left\{ \Theta _{\alpha _{k}}\right\}
\end{equation*}%
\begin{equation*}
\geq c_{p}\left( 2^{\left( \left\vert \alpha _{k}\right\vert -\left[ \alpha
_{k}\right] \right) \left( 1/p-2\right) }/\Phi \left( \alpha _{k}\right)
\right) ^{p/2}=c_{p}\left( 2^{d\left( \alpha _{k}\right) \left( 1/p-2\right)
}/\Phi \left( \alpha _{k}\right) \right) ^{p/2}\rightarrow \infty ,\text{\
as }k\rightarrow \infty .
\end{equation*}

By combining (\ref{121}) and first part of Theorem \ref{theorem2} (see also (%
\ref{77})) we have that $\left\Vert IV_{1}\right\Vert _{L_{p,\infty }}\leq
c_{p}<\infty ,$\ $\left\Vert IV_{2}\right\Vert _{L_{p,\infty }}\leq
c_{p}<\infty $\ and%
\begin{equation*}
\left\Vert \sigma _{\alpha _{k}}F\right\Vert _{L_{p,\infty }}^{p}\geq
\left\Vert IV_{3}\right\Vert _{L_{p,\infty }}^{p}-\left\Vert
IV_{2}\right\Vert _{L_{p,\infty }}^{p}-\left\Vert IV_{1}\right\Vert
_{L_{p,\infty }}^{p}\rightarrow \infty ,\text{ as }k\rightarrow \infty .
\end{equation*}

Theorem \ref{theorem2} is proved.
\end{proof}

\begin{proof}[\textbf{Proof of Theorem \protect\ref{theorem3}.}]
Let $F\in H_{1/2}$ and $2^{k}<n\leq 2^{k+1}.$ Then%
\begin{equation*}
\left\Vert \sigma _{n}F-F\right\Vert _{1/2}^{1/2}
\end{equation*}%
\begin{equation*}
\leq \left\Vert \sigma _{n}F-\sigma _{n}S_{2^{k}}F\right\Vert
_{1/2}^{1/2}+\left\Vert \sigma _{n}S_{2^{k}}F-S_{2^{k}}F\right\Vert
_{1/2}^{1/2}+\left\Vert S_{2^{k}}F-F\right\Vert _{1/2}^{1/2}
\end{equation*}%
\begin{equation*}
=\left\Vert \sigma _{n}\left( S_{2^{k}}F-F\right) \right\Vert
_{1/2}^{1/2}+\left\Vert S_{2^{k}}F-F\right\Vert _{1/2}^{1/2}+\left\Vert
\sigma _{n}S_{2^{k}}F-S_{2^{k}}F\right\Vert _{1/2}^{1/2}
\end{equation*}%
\begin{equation*}
\leq c_{p}\left( V^{2}\left( n\right) +1\right) \omega
_{H_{1/2}}^{1/2}\left( 1/2^{k},F\right) +\left\Vert \sigma
_{n}S_{2^{k}}F-S_{2^{k}}F\right\Vert _{1/2}^{1/2}.
\end{equation*}

By simple calculation we get that%
\begin{equation*}
\sigma _{n}S_{2^{k}}F-S_{2^{k}}F=\frac{2^{k}}{n}\left( S_{2^{k}}\sigma
_{2^{k}}F-S_{2^{k}}F\right) =\frac{2^{k}}{n}S_{2^{k}}\left( \sigma
_{2^{k}}F-F\right) .
\end{equation*}

Let $p>0.$ Then (see inequality (\ref{77}))
\begin{equation*}
\left\Vert \sigma _{n}S_{2^{k}}F-S_{2^{k}}F\right\Vert _{1/2}^{1/2}
\end{equation*}%
\begin{equation*}
\leq \frac{2^{k/2}}{n^{1/2}}\left\Vert S_{2^{k}}\left( \sigma
_{2^{k}}F-F\right) \right\Vert _{1/2}^{1/2}\leq \left\Vert \sigma
_{2^{k}}F-F\right\Vert _{H_{1/2}}^{1/2}\rightarrow 0,\text{ as }k\rightarrow
\infty \text{.}
\end{equation*}

Now, proof second part of theorem \ref{theorem3}. Since $\sup_{k\in \mathbb{N%
}}(\alpha _{k})=\infty ,$ there exists sequence $\{\alpha _{k}:k\geq
1\}\subset \{n_{k}:k\geq 1\}$ such that $V(\alpha _{k})\uparrow \infty
,\,\,\,\,$as $k\rightarrow \infty $ and%
\begin{equation}
V^{2}(\alpha _{k})\leq V(\alpha _{k+1}).  \label{4.7}
\end{equation}

We set
\begin{equation*}
F_{A}=\sum_{\left\{ i:\text{ }\left\vert \alpha _{i}\right\vert <A\right\}
}a_{i}^{1/2}/V^{2}(\alpha _{i}).
\end{equation*}

Since $a_{i}^{1/2}(x)$ is 1/2-atom if we apply Lemma \ref{lemma0} and (\ref%
{4.7}) we conclude that $F\in H_{1/2}.$ Moreover,%
\begin{equation}
F-S_{2^{n}}F=\left(
F_{1}-S_{2^{n}}F_{1},...,F_{n}-S_{2^{n}}F_{n},...,F_{n+k}-S_{2^{n}}F_{n+k}%
\right)  \label{20}
\end{equation}%
\begin{equation*}
=\left(
0,...,0,F_{n+1}-S_{2^{n}}F_{n+1},...,F_{n+k}-S_{2^{n}}F_{n+k},...\right)
\end{equation*}%
\begin{equation*}
=\left( 0,...,0,\underset{i=n}{\overset{n+k}{\sum }}a_{i}^{1}/V^{2}(\alpha
_{i}),...\right) ,\ k\in \mathbb{N}_{+}
\end{equation*}%
is martingale and by combining (\ref{20}) and Lemma \ref{lemma0} we get that%
\begin{equation}
\left\Vert F-S_{2^{n}}F\right\Vert _{H_{1/2}}\leq
\sum\limits_{i=n+1}^{\infty }1/V^{2}(\alpha _{i})=O\left( 1/V^{2}(\alpha
_{n})\right) .  \label{4.12}
\end{equation}

It is easy to show that%
\begin{equation}
\widehat{F}(j)=\left\{
\begin{array}{ll}
2^{_{\left\vert \alpha _{k}\right\vert }}/V^{2}(\alpha _{k}), & \text{%
\thinspace }j\in \left\{ 2^{\left\vert \alpha _{k}\right\vert
},...,2^{_{\left\vert \alpha _{k}\right\vert +1}}-1\right\} ,\text{ }%
k=0,1,... \\
0\,, & j\notin \bigcup\limits_{k=0}^{\infty }\left\{ 2^{_{\left\vert \alpha
_{k}\right\vert }},...,2^{_{\left\vert \alpha _{k}\right\vert +1}}-1\right\}
.\text{ }%
\end{array}%
\right.  \label{4.13}
\end{equation}

Let $2^{_{\left\vert \alpha _{k}\right\vert }}<j\leq \alpha _{k}.$ By using (%
\ref{f1})\ we have that
\begin{equation*}
S_{j}F=S_{2^{_{\left\vert \alpha _{k}\right\vert
}}}F+\sum_{v=2^{_{\left\vert \alpha _{k}\right\vert }}}^{j-1}\widehat{F}%
(v)w_{v}=S_{2^{_{\left\vert \alpha _{k}\right\vert }}}F+\frac{%
2^{_{\left\vert \alpha _{k}\right\vert }}w_{2^{_{\left\vert \alpha
_{k}\right\vert }}}D_{_{j-2^{_{\left\vert \alpha _{k}\right\vert }}}}}{%
V^{2}(\alpha _{k})}.
\end{equation*}%
Hence%
\begin{equation}
\sigma _{\alpha _{k}}F-F=\frac{2^{_{\left\vert \alpha _{k}\right\vert }}}{%
\alpha _{k}}\left( \sigma _{2^{_{\left\vert \alpha _{k}\right\vert
}}}(F)-F\right)  \label{nn}
\end{equation}%
\begin{equation*}
+\frac{\alpha _{k}-2^{_{\left\vert \alpha _{k}\right\vert }}}{\alpha _{k}}%
\left( S_{2^{_{\left\vert \alpha _{k}\right\vert }}}F-F\right) +\frac{%
2^{_{\left\vert \alpha _{k}\right\vert }}w_{2^{_{\left\vert \alpha
_{k}\right\vert }}}\left( \alpha _{k}-2^{_{\left\vert \alpha _{k}\right\vert
}}\right) K_{\alpha _{k}-2^{_{\left\vert \alpha _{k}\right\vert }}}}{\alpha
_{k}V^{2}(\alpha _{k})}.
\end{equation*}

By applying (\ref{nn}) we get that%
\begin{equation}
\Vert \sigma _{\alpha _{k}}F-F\Vert _{1/2}^{1/2}\geq \frac{c}{V(\alpha _{k})}%
\Vert \left( \alpha _{k}-2^{_{\left\vert \alpha _{k}\right\vert }}\right)
K_{\alpha _{k}-2^{_{\left\vert \alpha _{k}\right\vert }}}\Vert _{1/2}^{1/2}
\label{a11}
\end{equation}%
\begin{equation*}
-\left( \frac{2^{_{\left\vert \alpha _{k}\right\vert }}}{\alpha _{k}}\right)
^{1/2}\Vert \sigma _{2^{_{\left\vert \alpha _{k}\right\vert }}}F-F\Vert
_{1/2}^{1/2}-\left( \frac{\alpha _{k}-2^{_{\left\vert \alpha _{k}\right\vert
}}}{\alpha _{k}}\right) ^{1/2}\Vert S_{2^{_{\left\vert \alpha
_{k}\right\vert }}}F-F\Vert _{1/2}^{1/2}.
\end{equation*}

Let $\alpha _{k}=\sum_{i=1}^{r_{k}}\sum_{k=l_{i}^{k}}^{m_{i}^{k}}2^{k},$
where $m_{1}^{k}\geq l_{1}^{k}>l_{1}^{k}-2\geq m_{2}^{k}\geq
l_{2}^{k}>l_{2}^{k}-2>...>m_{s}^{k}\geq l_{s}^{k}\geq 0$ and $%
E_{l_{i}^{k}}:=I_{_{l_{i}^{k}+1}}\left( e_{l_{i}^{k}-1}+e_{l_{i}^{k}}\right)
.$ By using Lemma \ref{lemma3} we have that%
\begin{equation}
\int_{G}\left\vert \left( \alpha _{k}-2^{_{\left\vert \alpha _{k}\right\vert
}}\right) K_{\alpha _{k}-2^{_{\left\vert \alpha _{k}\right\vert }}}\left(
x\right) \right\vert ^{1/2}d\mu  \label{33}
\end{equation}%
\begin{equation*}
\geq \frac{1}{16}\underset{i=2}{\overset{s_{k}-2}{\sum }}%
\int_{E_{l_{i}^{k}}}\left\vert \left( \alpha _{k}-2^{_{\left\vert \alpha
_{k}\right\vert }}\right) K_{\alpha _{k}-2^{_{\left\vert \alpha
_{k}\right\vert }}}\left( x\right) \right\vert ^{1/2}d\mu \left( x\right)
\geq \frac{1}{16}\underset{i=2}{\overset{s_{k}-2}{\sum }}\frac{1}{%
2^{l_{i}^{k}}}2^{l_{i}^{k}}
\end{equation*}%
\begin{equation*}
\geq cs_{k}\geq cV(\alpha _{k}).
\end{equation*}

By combining (\ref{a11}-\ref{33}) we have (\ref{kn3}) and Theorem \ref%
{theorem3} is proved.
\end{proof}

\begin{proof}[\textbf{Proof of Theorem \protect\ref{theorem4}.}]
\textbf{\ }Let\textbf{\ }$0<p<1/2.$ Under condition (\ref{cond3}), if we
repeat analogical steps of first part of Theorem \ref{theorem3}, we can show
that (\ref{fe2}) holds.

Now, prove second part of theorem \ref{theorem4}. Since $\sup_{k}d\left(
n_{k}\right) =\infty ,$ there exists $\{\alpha _{k}:k\geq 1\}\subset
\{n_{k}:k\geq 1\}$ such that $\sup_{k}d\left( \alpha _{k}\right) =\infty $
and
\begin{equation}
2^{2d\left( \alpha _{k}\right) \left( 1/p-2\right) }\leq 2^{d\left( \alpha
_{k+1}\right) \left( 1/p-2\right) }.  \label{4.18}
\end{equation}%
We set
\begin{equation*}
F_{A}=\sum_{\left\{ i:\text{ }\left\vert \alpha _{i}\right\vert <A\right\}
}a_{i}^{\left( p\right) }/2^{\left( 1/p-2\right) d\left( \alpha _{i}\right)
}.
\end{equation*}

Since $a_{i}^{p}(x)$ is $p$-atom if we apply Lemma \ref{lemma0} and (\ref%
{4.18}) we conclude that $F\in H_{p}.$ Analogously to (\ref{20}) we can show
that%
\begin{equation}
F-S_{2^{\left\vert \alpha _{k}\right\vert }}F=\left( 0,...,0,\underset{i=k}{%
\overset{k+s}{\sum }}a_{i}/2^{d\left( \alpha _{i}\right) \left( 1/p-2\right)
},...\right) ,\text{ \ }s\in \mathbb{N}_{+}  \label{78}
\end{equation}%
is martingale. By combining (\ref{4.18}) and Lemma \ref{lemma0} we get that%
\begin{equation}
\omega _{H_{p}}(1/2^{\left\vert \alpha _{k}\right\vert },F)  \label{4.21}
\end{equation}%
\begin{equation*}
\leq \sum\limits_{i=k}^{\infty }1/2^{d\left( \alpha _{i}\right) \left(
1/p-2\right) }=O\left( 1/2^{d\left( \alpha _{k}\right) \left( 1/p-2\right)
}\right) .
\end{equation*}

It is easy to show that%
\begin{equation}
\widehat{F}(j)=\left\{
\begin{array}{ll}
2^{\left( 1/p-2\right) \left[ \alpha _{k}\right] }, & \text{\thinspace
\thinspace }j\in \left\{ 2^{\left\vert \alpha _{k}\right\vert
},...,2^{_{\left\vert \alpha _{k}\right\vert +1}}-1\right\} ,\text{ }%
k=0,1,... \\
0\,, & \text{\thinspace }j\notin \bigcup\limits_{n=0}^{\infty }\left\{
2^{_{\left\vert \alpha _{n}\right\vert }},...,2^{_{\left\vert \alpha
_{n}\right\vert +1}}-1\right\} .\text{ }%
\end{array}%
\right.  \label{4.22}
\end{equation}

Analogously we can write that%
\begin{equation*}
\Vert \sigma _{\alpha _{k}}F-F\Vert _{L_{p,\infty }}^{p}\geq 2^{\left(
1-2p\right) \left[ \alpha _{k}\right] }\Vert \left( \alpha
_{k}-2^{\left\vert \alpha _{k}\right\vert }\right) K_{\alpha
_{k}-2^{_{\left\vert \alpha _{k}\right\vert }}}\Vert _{L_{p,\infty }}^{p}
\end{equation*}%
\begin{equation*}
-\left( \frac{2^{\left\vert \alpha _{k}\right\vert }}{\alpha _{k}}\right)
^{p}\Vert \sigma _{2^{_{\left\vert \alpha _{k}\right\vert }}}F-F\Vert
_{L_{p,\infty }}^{p}-\left( \frac{\alpha _{k}-2^{_{\left\vert \alpha
_{k}\right\vert }}}{\alpha _{k}}\right) ^{p}\Vert S_{2^{_{\left\vert \alpha
_{k}\right\vert }}}F-F\Vert _{L_{p,\infty }}^{p}.
\end{equation*}

Let $x\in E_{\left[ \alpha _{k}\right] }.$ By combining (\ref{9a}) and Lemma %
\ref{lemma1} we have that%
\begin{equation*}
\mu \left( x\in G:\left( \alpha _{k}-2^{_{\left\vert \alpha _{k}\right\vert
}}\right) \left\vert K_{\alpha _{k}-2^{_{\left\vert \alpha _{k}\right\vert
}}}\right\vert \geq 2^{2\left[ \alpha _{k}\right] -4}\right) \geq \mu \left(
E_{\left[ \alpha _{k}\right] }\right) \geq 1/2^{\left[ \alpha _{k}\right]
-4},
\end{equation*}%
\begin{equation*}
2^{2p\left[ \alpha _{k}\right] -4}\mu \left( x\in G:\left( \alpha
_{k}-2^{_{\left\vert \alpha _{k}\right\vert }}\right) \left\vert K_{\alpha
_{k}-2^{_{\left\vert \alpha _{k}\right\vert }}}\right\vert \geq 2^{2\left[
\alpha _{k}\right] -4}\right) \geq 2^{\left( 2p-1\right) \left[ \alpha _{k}%
\right] -4}.
\end{equation*}

Hence
\begin{equation*}
\left\Vert \sigma _{n_{k}}F-F\right\Vert _{L_{p,\infty }}\nrightarrow
0,\,\,\,\text{as}\,\,\,k\rightarrow \infty
\end{equation*}%
and Theorem \ref{theorem4} is proved.
\end{proof}

\end{document}